 \documentclass[preprint,12pt]{elsarticle}

\usepackage{amssymb,amsmath,amsthm}
\usepackage{booktabs}  %for table creation
\usepackage{multirow}
\usepackage{rotating}
\usepackage{epsf,epstopdf}
\usepackage{graphicx,graphics}
\usepackage{color}
\usepackage[skip=2pt,font = scriptsize]{caption}
\usepackage{float}
\usepackage{adjustbox}
\numberwithin{equation}{section}

\newtheorem{thm}{Theorem}

\newtheorem{eg}{Example}
\newtheorem{lem}{Lemma}

\def\vs{\vspace}
\journal{TBA}

\begin{document}

\begin{frontmatter}



%\title{Augmented GMRES with Singular Vector Approximations}
\title{GMRES with Singular Vector Approximations}


\author{Mashetti Ravibabu\corref{cor1}}
%Someshwara
\cortext[cor1]{Corresponding author.}
\address{Dept of Computational and Data Sciences, Indian Institute of Science,
Bengaluru, India-560012.\\ Email: mashettiravibabu2@gmail.com}

%\author{Arindama Singh\corref{cor1}}
%\cortext[cor1]{Corresponding author.}
%\address{Dept of Mathematics, Indian Institute of Technology
%Madras, Chennai, India-600036. Email: asingh@iitm.ac.in}

\begin{abstract}
This paper has proposed the GMRES that augments Krylov subspaces with a set of approximate right singular vectors.
%GMRES that augments Krylov subspaces with a set of approximate singular vectors has proposed in this paper. 
The proposed method suppresses the error norms 
%in addition to residual norms 
of a linear system of equations. Numerical experiments comparing the proposed method with the Standard GMRES and GMRES with eigenvectors methods\cite{aug}  have been reported for benchmark matrices.
%It is well-known that at each iteration the GMRES method reduces residual norm associated with the linear system of equations. In this paper, we propose an augmentation of a search subspace in the GMRES method with singular vector
%approximations. The proposed
%method has an advantage that  apart from reducing residual norms it also reduces corresponding error norms.  Numerical
%experiments have been carried out on some benchmark matrices and results have been compared with the standard GMRES and the GMRES with
%eigenvectors proposed in \cite{aug}. Numerical results have been demonstrated that the error
%reduction feature of the proposed method accelerates the convergence of approximate solutions in the GMRES.
\end{abstract}


\end{frontmatter}



\section{Introduction}
The 
%Generalized minimal residual method, in short 
GMRES method is well-known for solving $Ax=b,$ a large sparse system of linear equations, especially, when an approximate solution is sufficient \cite{saad}. 
%Residual norms of approximate solutions in the  GMRES are relatively smaller as the residual norm controls the progress of its iterations. 
Nonetheless, the error norms of approximate solutions in GMRES need not be smaller; See \cite[Example-1]{Gmrerr}.  
%The residual norm controls the progress of iterations in the GMRES method.
%progress of iteration in the GMRES method is controlled by the residual norm. It results in approximate solutions with smaller residual norms in the GMRES. %\textcolor{red}{Thus, residual norms of  in the GMRES are smaller in magnitude.} 
%However, the corresponding error norms may not be small; See Example-1 in \cite{Gmrerr}.  The larger error norms in the GMRES may restrict its use in solving the linear system of equations emanating from few applications. For example, in 
%the Jacobi-Davidson method, 
%the inexact Arnoldi and the inexact Rayleigh quotient iteration methods for solving large sparse eigenvalue problems. 
%In this paper, we develop a variant of GMRES that reduce error norms to a better extent than standard GMRES.
%Many iterative methods in the literature minimize the error norm when the coefficient matrix $A$ is symmetric.
%Several iterative methods exist in the literature those minimize the error norm when the coefficient matrix $A$ is symmetric. 
%For the non-symmetric $A,$ 
%When the coefficient matrix $A$ is symmetric, several iterative methods
%were proposed and discussed in the literature for minimizing the
%error norm. For the non-symmetric case, 
Prompted by this, Weiss proposed an algorithm that minimizes error norms,
%Weiss to propose an error minimizing algorithm, 
The Generalized Minimal Error method (GMERR)%, a method analog to GMRES
\cite{errm}. Ehrig and Deuflhard studied convergence properties of GMERR and developed an algorithm that has an implementation similar to GMRES \cite{Gmrerr}. 
%However, their numerical experiments have confirmed that the performance of GMERR is not competitive to GMRES for non-normal matrices. 

Later, A stable variant of GMERR proposed using Householder transformations and has observed that the full version of GMERR may be effective in reducing the error, but its performance is not competitive to GMRES 
%the restarting variant of GMERR may competentness with restarting GMRES, but not the full version of GMERR with its counterpart of the GMRES. 
%the error minimizing algorithms overall 
\cite{stab}. Although CGNR minimizes both error and residual norms, its convergence depends on the square of the condition number of $A$ \cite{CGNR}.

This paper develops a tool that can combine to standard restarting GMRES, and reduce both error and residual norms. The tool augments the Krylov subspaces in restarting GMRES with a set of approximate right singular vectors. %This tool is analog of that used in \cite{aug}.  
The following is the outline of this paper: In Section-2, we present the
GMRES method. Section-3 develops the said tool, and analyzes the convergence properties of the GMRES with approximate Singular vectors method. Section-4 gives implementation details of the algorithm proposed in section-3. Section-5 reports the results of numerical experiments on some benchmark matrices. Section-6 concludes the paper.

%The following is the outline of this paper: In Section-2, we present the GMRES method.
%% and review the  reasons for restarting it after a few number of iterations.
% Section-3 proposes a method to augment a Krylov subspace in the GMRES with a vector space of singular vector approximations. Then, it will
%analyze the convergence properties of the proposed method. Section-4 discusses implementation details of the
%algorithm proposed in the Section-3. Section-5 reports the results of numerical
%experiments on some benchmark matrices. Section-6 concludes the paper.
%augmentation of GMRES with singular vector approximations and will
%analyze convergence properties of GMRES. Implementation details of
%algorithm will be discussed in Section-4. Results of numerical
%experiments on some benchmark matrices are reported in Section-5.
\section{GMRES}\label{gmres}
Consider the following system of linear equations:
$$Ax =b, A \in \mathcal{C}^{~n \times n}, b~\in~\mathcal{C}^{~n}, x \in \mathcal{C}^{~n}.$$
Let $x0 \in \mathcal{C}^n$ be an arbitrarily chosen initial
approximation to the solution of the above problem. Without loss of
generality, it is assumed throughout the paper that $x0=0$ so that
$r_0 =b.$ Then, at $i^{th}$ iteration of GMRES an approximate
solution belongs to the Krylov subspace:
$$\mathcal{K}_i(A,b) ~=~span\{b,Ab,\cdots, A^{i-1} b\},$$
and is of minimal residual norm:
\begin{equation}\label{eq1}\relax
\|r_i\| = \min_{x \in \mathcal{K}_i(A,b)} \|b-Ax\|.
\end{equation}
The GMRES method solves this minimization problem by generating the matrix $V_i=\begin{bmatrix} v_1~ v_2~ \cdots~ v_i \end{bmatrix}$ in successive iterations, using the following recurrence relation:
%
%
%The GMRES method solves this minimization problem in two steps by using the Arnoldi method and the least squares technique.
%%and involves two steps: 
%In the first step the Arnoldi method generates the matrix $AV_i$
%where $V_i=\begin{bmatrix} v_1~ v_2~ \cdots~ v_i \end{bmatrix},$ in successive iterations by using the following recurrence relation:

% and
%$v_j$ for $j=1,2,\cdots i$ forms an orthonormal basis for the Krylov subspace
%$\mathcal{K}_i(A,b).$ In GMRES method, this matrix $AV_i$ is
%generated in successive iterations by using the following Arnoldi
%recurrence relation:
\begin{equation}\label{eq1a}\relax
AV_i=V_iH_i+h_{i+1,i}v_{i+1}e_i^\ast, ~~\mbox{where}~~v_1 =
\frac{b}{\|b\|},
\end{equation}
and $H_i$ is an unreduced upper Hessenberg matrix of order
$i.$ Here, a vector $v_{i+1} \perp v_j$ for $j = 1,2,\cdots i,$ and 
$\|v_{i+1}\|=1.$  Thus, vectors $v_j$ for $j=1,2,\cdots i$ form an orthonormal basis for the Krylov subspace $\mathcal{K}_i(A,b).$

Next, by using the equation (\ref{eq1a}), GMRES recasts the minimization problem in (\ref{eq1}) into the following:
%
%In the second step GMRES solves the following least squares
%problem which is theoretically equivalent to the minimization problem in the equation (\ref{eq1}):
%$$z_i = \arg\min_{x \in \mathcal{C}^i} \|b-AV_ix\|. $$
%Thus, vector $z_i$ is a solution of following  normal system of
%equations:
%\begin{equation}\label{eq2}\relax
%V_i^\ast A^\ast AV_iz_i = V_i^\ast A^\ast b = \beta V_i^\ast A^\ast
%V_ie_1
%\end{equation}
%where $\beta = \|b\|.$ Since columns of $V_i$ are orthonormal and
%$v_{i+1}$ is orthogonal to columns of $V_i,$ above equation can be
%written in terms of upper Hessenberg matrix $H_i$ as follows:
%\begin{equation}\label{eq3}\relax
%(H_i^\ast H_i+|h_{i+1,i}|^2 e_ie_i^\ast)z_i = \beta H_i^\ast z_i.
%\end{equation}
%By using the equation (\ref{eq1a}), the GMRES transforms the above least squares problem into the following:
$$z_i = \arg\min_{x \in \mathcal{C}^i} \|b-AV_ix\| = \arg\min_{x \in \mathcal{C}^i} \|\beta V_{i+1}e_1-V_{i+1}\tilde{H_i} x\|,$$
where $\beta = \|b\|,$ and $\tilde{H_i}$ is an upper Hessenberg matrix 
%of order $(i+1) \times i$ 
obtained by appending the row $[0~0~\cdots~h_{i+1,i}]$ at the bottom of the matrix $H_i.$ As columns of the matrix $V_{i+1}$ are orthonormal, the above least squares problem is equivalent to the following problem:
$$z_i =\arg\min_{x \in \mathcal{C}^i} \|\beta e_1-\tilde{H_i} x\|.$$
% 
%
%As columns of the matrix $V_{i+1}$ are orthonormal, the following are normal equations for the above least squares problem:
%%written in terms of upper Hessenberg matrix $H_i$ as follows:
%$$(H_i^\ast H_i+|h_{i+1,i}|^2 e_ie_i^\ast)z_i = \beta H_i^\ast z_i.$$
GMRES solves this problem for the vector $z_i$ by using the $QR$ decomposition of the matrix $\tilde{H_i}.$ Note that a vector $V_iz_i$ minimizes the associated residual norm due to the equation~(\ref{eq1}). Thus, the sequence of residual norms $\{\|b-AV_iz_i\|_2\}$ in GMRES is monotonically decreasing. However, the corresponding sequence of error norms  may not decrease monotonically. To tackle this, in the following sections, we augment the Krylov subspace in the GMRES method with a vector space containing approximate singular vectors. 
%
%indicates that norm of residual vectors in GMRES is monotonically
%decreasing. The norm of residual will be zero, i.e., GMRES gives
%exact solution when $\mathcal{K}_{d+1}(A,b) =\mathcal{K}_d(A,b).$
%Here, $d$ is the degree of a minimal polynomial of a vector $b$ with
%respect to $A.$ If complete stagnation occurs, then $d=n,$ order of
%matrix $A,$ since in this case $r_0=r_1=\cdots=r_{n-1} \neq 0$ or
%$\|r_0\|=\|r_1\|=\cdots=\|r_{n-1}\|\neq 0.$ Note that at the
%$n^{th}$ iteration, matrix $V:=V_n$ is unitary and matrices $A,~V$
%satisfies the relation:
%\begin{equation}\label{eq4}\relax
%AV =VH,
%\end{equation}
%where $H:=H_n,$ is an unreduced upper Hessenberg matrix of order
%$n.$
\section{Motivation:}
In this section, we are addressing the augmentation of a Krylov subspace in  GMRES with the singular vectors. The following lemma explains the motivation behind this. 
\begin{lem}
Let $z$ be a right singular vector of a matrix $A$ corresponding to the 
singular value $\sigma,$ and $x0$ be an approximate solution of a
linear system of equations $Ax=b.$ Then, a solution of the following minimization problem
\begin{equation}\label{mp1}\relax
\alpha = \arg \min_{k \in \cal{C}}\|b-Ax0-kAz\|,
\end{equation}
is the solution of the minimization problem 
\begin{equation}\label{mp2}\relax
\alpha = \arg \min_{k \in \cal{C}}\|x-x0-kz\|.
\end{equation}
\end{lem}
\begin{proof}
Let $\alpha$ be a
solution of the minimization problem (\ref{mp1}). By using $Ax=b,$
%$\|b-Ax0-k Az\|^2 = \|Ax-Ax0-k Az\|^2~\forall k,$ 
this implies $\alpha = \frac{\langle Ax-Ax0,Az \rangle}{\|Az\|^2}.$ Further, by using $\|z\|_2=1$ and  $A^\ast A z =\sigma^2 z$ from the hypothesis, we have $\alpha =
\frac{\langle x-x0,z \rangle}{\|z\|^2}.$ 
Equivalently, this gives
%can be written as 
$$\langle x-x0-\alpha z,z \rangle = \langle x-x0,z \rangle- \alpha
\|z\|^ 2=0.$$ Thus, a vector
$x-x0-\alpha z$ is orthogonal to the vector space spanned by the vector $z.$ Therefore, $\alpha$ is a
solution of the  minimization problem (\ref{mp2}).
\end{proof}
The above lemma has shown an advantage of the singular vectors that they reduce both residual and error norms.
% and also the norm of the corresponding error vector. 
Now the following two questions arise when augmenting the Krylov subspace in GMRES with a vector space spanned by singular vectors.
The first question is computing a singular vector of a sparse matrix requires more computation than finding the solution of a sparse linear system of equations, in general. The second question to address is singular vectors corresponding to what singular values are better to augment Krylov subspaces in GMRES. 
%We will make use of Krylov subspaces generated in GMRES to get these approximations. 

Usage of an approximate singular vector in the augmentation process instead of an exact singular vector resolves the first problem. The following theorem discusses the effect of this usage.
%replacing the exact singular vectors in the augmentation process.% by their approximations.
%In the following theorem, we discuss the effect of replacement of exact singular vectors by its approximations in the augmentation process with Krylov subspaces.

%In general, computing a singular vector of a sparse matrix
%requires more computation than finding the solution of a sparse linear 
%system of equations. This problem can be resolved by using an 
%approximate singular vector instead of an exact singular vector in the augmentation process. To get these approximations, we will make use of Krylov subspaces generated in the GMRES. 
%The second question to address is that singular
%vectors corresponding to what singular values are better to be
%augmented with Krylov subspaces. First, In the following Theorem, we
%discuss the affect of replacement of exact singular vectors with its
%approximations in the augmentation process with Krylov subspaces.
\begin{thm}\label{mg}\relax
Let a matrix $V_m$ have orthonormal columns and $x0$ be an
approximate solution of a linear system of equations $Ax =
b.$ Assume that $x-x0$ is in the range space of $V_m$ and $z$ is a right
singular vector of the matrix $AV_m$ corresponding to its singular value
$\sigma.$ Then a solution of the following minimization
problem
\begin{equation}\label{mp3}\relax
\alpha = \arg \min_{k \in \cal{C}}\|b-Ax0-kAV_mz\|
\end{equation}
is the solution of the minimization problem:
\begin{equation}\label{mp4}\relax
\alpha = \arg \min_{k \in \cal{C}}\|x-x0-kV_mz\|.
\end{equation}
\end{thm}
\begin{proof}
Let $x-x0=V_my.$ Note that if $y$ is a
zero vector then $x0$ is an exact solution of the linear system $Ax=b.$ Assume that $y$ is a non-zero vector. Now, $\alpha,$ a solution of the minimization problem (\ref{mp3}) is 
\begin{equation}\label{nt}\relax
\alpha
= \frac{\langle Ax-Ax0,AV_mz\rangle}{\|AV_mz\|^2}=\langle y,z\rangle=\langle x-x0,V_mz
\rangle.
\end{equation}
The above equation has used the facts that $V_m^\ast A^\ast A V_m z= \sigma ^2 z,$ and $V_m^\ast V_m$ is an Identity matrix. Therefore, by using the
same lines of proof as in the previous lemma, $\alpha$ is a solution of the error minimization problem (\ref{mp4}).
\end{proof}
The Theorem-\ref{mg} says that if $x-x0$ is in the Krylov subspace spanned by the columns of $V_m$ a singular vector of $AV_m$ will serve the purpose of a singular vector of $A$ in the augmentation process. However, the error vector $x-x0$ may not lie entirely in the said
subspace,
% spanned by columns of $V_m,$ 
in general. In this case, the following theorem establishes a relationship between the solutions of minimization problems (\ref{mp3}) and (\ref{mp4}).
\begin{thm}\label{lem2}\relax
Let $x0$ be an approximate solution of the linear system of equations $Ax=b$ and $\sigma,z$ are same as in the previous theorem. Assume that $\alpha_1,$ $\alpha_2$ are solutions of the minimization problems (\ref{mp3}) and (\ref{mp4}) respectively. Then, 
\begin{equation}\label{a12}\relax
\|x-x0-\alpha_1 V_mz\|^2-\|x-x0-\alpha_2V_mz\|^2= \frac{|\langle x-x0, (A^\ast
A-\sigma^2 I)V_mz \rangle|^2}{\sigma^4}.
\end{equation}
\end{thm}
\begin{proof}
Note that $x-x0=V_mV_m^\ast(x-x0)+(I-V_mV_m^\ast)(x-x0).$ By using $Ax=b$ and this, the solution $\alpha_1$ of the minimization problem (\ref{mp3}) can be
written as 
$$\alpha_1 = \frac{\langle AV_mV_m^\ast (x-x0)+A(I-V_mV_m^\ast )(x-x0), AV_mz
\rangle}{\|AV_mz\|^2}. $$
On substituting the equation (\ref{nt}) this yields
$$\alpha_1= \langle x-x0,V_mz \rangle +\frac{\langle
x-x0, (I-V_mV_m^\ast)A^\ast AV_mz \rangle}{\|AV_mz\|^2}.$$
Further, by using 
$V_m^\ast A^\ast AV_mz = \sigma^2 z,$ this gives
\begin{equation}\label{a1}\relax
\alpha_1 = \langle x-x0,V_mz \rangle +\frac{\langle x-x0, (A^\ast A- \sigma^2 I) V_m z \rangle}{\|AV_mz\|^2}=\big\langle x-x0,\frac{A^\ast AV_mz}{\sigma^2} \big\rangle.
\end{equation}
The above equation used the fact that $\|AV_mz\|^2 = \sigma^2.$ 
As $\|V_mz\|_2=1$ and $\alpha_2$ is the solution of an error minimization problem (\ref{mp4}), we have $\alpha_2= \langle x-x0,V_mz \rangle,$ and 
$$\|x-x0-\alpha_1 V_mz\|^2-\|x-x0-\alpha_2V_mz\|^2 =
|\alpha_1-\alpha_2|^2\|V_mz\|^2 = |\alpha_1-\alpha_2|^2.$$ Now, observe from the equation (\ref{a1}) that 
\begin{equation}\label{a1m2}\relax
\alpha_1-\alpha_2 = \big\langle x-x0,\frac{(A^\ast A-\sigma^2 I)
V_mz}{\sigma^2} \big \rangle,
\end{equation}
and substitute it in the previous equation. It gives the equation (\ref{a12}). Hence, we proved the theorem.
\end{proof}

%The equation (\ref{a1m2}) says that $|\alpha_1-\alpha_2|$ may be small, when $V_mz$ is a good approximation to the right singular vector of $A$ or the corresponding singular value $\sigma^2$ is large. 
From $V_m^\ast A^\ast AV_mz=\sigma^2z$ note that $(A^\ast A-\sigma^2 I)
V_mz=(I-V_mV_m^\ast)(A^\ast A-\sigma^2 I)
V_mz.$ On substituting this in the right-hand side of the equation (\ref{a12}), it is easy to see that 
the difference between  $\|x-x0-\alpha_1 V_mz\|^2$ and $\|x-x0-\alpha_2 V_mz\|^2$ was only due to components orthogonal to $V_m$ in $x-x0,$ that means, 
%In turn, again using the Thoerem -\ref{lem2}, we can say this statement as follows:
%the components orthogonal to column space of $V_m$ in $x-x0$ are prompting 
%the difference between $\|x-x0-\alpha_1 V_mz\|^2$ and $\|x-x0-\alpha_2 V_mz\|^2.$ In turn, using the Thoerem -\ref{lem2}, we can say the statement ``The difference between  $\|x-x0-\alpha_1 V_mz\|^2$ and $\|x-x0-\alpha_2 V_mz\|^2$ was only due to components orthogonal to $V_m$ in $x-x0"$ as follows:
"The components from the column space of $V_m$ are optimally balanced in the error vector $x-x0-\alpha_1V_mz."$ Thus, augmenting a search subspace in  GMRES with a singular vector approximation will accelerate the convergence of approximate solutions.
%Thus, unlike in $x-x0,$ the components from the column space of $V_m$ are optimally balanced in the error vector $x-x0-\alpha_1V_mz.$
This fact motivates us to augment the Krylov subspace at each run in the restarting GMRES with the singular vector approximations as explained in the following paragraph.
% from the previous run. The following paragraph explains this.

\vs{0.2cm}
Let $e_m^{(i)}$ denotes an error vector, and the columns of $V_m^{(i)}$ span the search subspace at the $i^{th}$ run of restarting GMRES. 
%Then recall that at the  
%$i^{th}$ run the restarting GMRES minimizaes the residual norm over the subspace spanned by  columns of $V_m^{(i)}.$ 
Suppose $z$ is a right singular vector of $AV_m^{(i)}$ and it augments the column space of $V_m^{(i+1)}$ at the $(i+1)^{th}$ run.
% corresponding to the  singular value 
 Then, the Theorem-\ref{lem2} and the previous paragraph says that 
%unlike in $e_m^{(i+1)}$ 
the components of the column space of $V_m^{(i)}$ are optimally balanced in the error vector that corresponds to the vector minimizing a residual norm 
%for which the residual norm is minimum 
over Range$(V_m^{(i+1)},V_m^{(i)} z).$ 
%Thus, augmenting search subspace at each run  in the restarting GMRES with a singular vector approximation from the previous run will accelerate the convergence of approximate solutions.

%The Theorems-\ref{mg} and \ref{lem2} answered the first question that we arose before. 
Next, the following theorem is required to answer the second question that  we arose before, approximate singular vectors corresponding to what singular values are better to augment the Krylov subspace in GMRES. 
%It explains about the convergence of residual norms in the augmented GMRES. 
The proof will follow the lines of Sections 3 and 4 in
\cite{zitko}.
%\begin{thm}\label{conv}\relax
%Let a subspace range of $Y_k$ be augmenting the Krylov subspace 
%$\mathcal{K}_m(A,r_0),$ 
%where $Y_k \in C^{n \times k}$ and $m+k < n.$ Let
%$z \in Y_k$ and  $y := \|r_0\|q(A)v+Az$ is a projection of the
%initial residual  $r_0=\|r_0\|v$ onto the \\space Range$(AV_m,
%AY_k).$ Then, an optimal residual $r_s:=r_0-y$ over the space \\Range$(AV_m,
%AY_k)$ satisfies  the following inequality:
%\begin{equation}\label{zit}\relax
%\frac{\|r_s\|^2}{\|r_0\|^2} \leq 1-\min_{w \in S_n} \frac{|w^\ast
%q(A) w|^2+\|w^\ast AY_k\|^2}{\|q(A)\|^2+\|AY_k\|_F^2},
%\end{equation}
%where $ S_n$ denotes the unit sphere in $C^n,$ $\|~.~\|_F$ is the Frobenius norm, and $q(A)$ is a polynomial in $A$ of degree $m$ such that $q(0)=0.$
%\end{thm}
\begin{thm}\label{conv}\relax
Let a subspace range of $Y_k$ be augmenting the Krylov subspace 
$\mathcal{K}_m(A,r_0),$ where $Y_k \in C^{n \times k}$ and $m+k < n.$ Assume that $z \in Y_k,$ and $q(A)$ is a polynomial in $A$ of degree $m$ such that $q(0)=0.$ Let $r_s:=r_0-y$ be an optimal residual over the space Range$(AV_m, AY_k), $ where $y := \|r_0\|q(A)v+Az.$ Then 
\begin{equation}\label{zit}\relax
\frac{\|r_s\|^2}{\|r_0\|^2} \leq 1-\min_{w \in S_n} \frac{|w^\ast
q(A) w|^2+\|w^\ast AY_k\|^2}{\|q(A)\|^2+\|AY_k\|_F^2},
\end{equation}
where $ S_n$ denotes the unit sphere in $C^n,$ and $\|~.~\|_F$ is the Frobenius norm.
\end{thm}
\begin{proof}
Let $U:= (q(A)v, AY_k)$ is a matrix of full rank. Then,
$P=U(UU^\ast)^{-1}U^\ast$ defines an orthogonal projection onto the
space $\textbf{Range}(q(A)v, AY_k).$ Thus, $$Pr_0 \in \textbf{Range}(q(A)v, AY_k).$$ Since $q(A)v \in AV_m,$ and $r_s:=r_0-y$ is an optimal residual over $\textbf{Range} (AV_m, AY_k), $
%orthogonal to the range of $(AV_m, AY_k)$ 
this implies
$$\|r_s\|^2 \leq \|r_0-Pr_0\|^2 = \|(I-P)r_0\|^2.$$
This gives
$$\frac{\|r_s\|^2}{\|r_0\|^2} \leq \|(I-P)\frac{r_0}{\|r_0\|}\|^2=\|(I-P)v\|^2 = 1-v^\ast Pv.$$
By using $P=U(UU^\ast)^{-1}U^\ast$ the above equation gives the following:
$$\frac{\|r_s\|^2}{\|r_0\|^2} \leq 1-\|U^\ast v\|^2 \lambda_{min}(U^\ast U)^{-1} \leq 1-\frac{\|U^\ast v\|^2}{\lambda_{max}(U^\ast U) }\leq 1-\frac{\|U^\ast v\|^2}{Trace(U^\ast U)}.$$
Since $U= (q(A)v, AY_k),$ we have $\|U^\ast v\|^2 =|v^\ast
q(A) v|^2+\|v^\ast AY_k\|^2$ and $Trace(U^\ast U)=\|q(A)v\|^2+\|AY_k\|_F^2.$
Substituting these inequalities in the above equation gives the following:
$$\frac{\|r_s\|^2}{\|r_0\|^2} \leq 1-\frac{|v^\ast
q(A) v|^2+\|v^\ast AY_k\|^2}{\|q(A)v\|^2+\|AY_k\|_F^2} \leq 1-\frac{|v^\ast
q(A) v|^2+\|v^\ast AY_k\|^2}{\|q(A)\|^2+\|AY_k\|_F^2}.$$
Here, the second inequality used the facts that $\|v\|_2=1$ and $\|q(A)v\|_2 \leq \|q(A)\|_2.$ Now minimizing the numerator of the second term over $S_n,$  the unit sphere in $C^n,$ gives the equation (\ref{zit}).
%  With the substitution of this and then by using the definition of Euclidean norm, the above equation yields the required inequality in (\ref{zit}). 
Hence, the theorem proved.
%\begin{eqnarray}
%  \leq  1-\min_{w \in S_n} \frac{|w^\ast q(A) w|^2+\|w^\ast
%AY_k\|^2}{\|q(A)\|^2+\|AY_k\|_F^2}~~~~~~~~~~~ ~~
%\end{eqnarray}
%Now, the proof is over by observing the first and last expressions
%in the above equation.
\end{proof}
%In the Theorem-\ref{conv} note that $y$ is an approximate solution of $Ax=b$ in the GMRES from the Krylov subspace augmented with range of $Y_k.$ 
 From the Theorem-\ref{conv} note that the norm of an updated residual
$\|r_s\|$ deviates more from $\|r_0\|$ when
% $Y_k$ is chosen such that 
$\|AY_k\|_F$ is smaller. It is well known that $\|AY_k\|_F$ is small when columns of $Y_k$ are right singular vectors corresponding to smaller singular
values of $A.$ This answers the second question that we arose before. 
%
%Thus, augmenting search subspace in the GMRES with approximate right singular vectors may effectively reduce residual norms . This motivates us to devise the GMRES method augmented with approximate singular vectors in the next section.
The next section devises the GMRES with approximate singular vectors method. The new algorithm  augments a Krylov subspace at each run with an approximate right singular vector from the previous run. 
%Though theoretically, an initial residual at each run is orthogonal to the vector augmented, 
%Thus, theoretically the Theorem-\ref{conv} may not applicable. However, 
%in practice, it is not true 
%residual is not orthogonal to $A$ image of a subspace in the previous run, 
%due to loss of orthogonality among Arnoldi vectors. 
%Further, in general, as an  error in the solution of a least squares problem has larger components 
%%occurs generally more
%in the direction of singular vectors corresponding to smaller singular values, % that we are using for augmentation,
%practically the above theorem shows the viability of the 
%%is applicable for the
%algorithm that we introduce in the following section.
\section{Implementation}
%The implementation procedure describing here for the new algorithm is similar
%to the one used by Morgan \cite{aug}. 
Let $x0$ be an initial approximate solution of a linear
system of equations $Ax =b,$ and $r_0=b-Ax0.$ Let $m$ be the
dimension of a search subspace consists of $m-k$ dimensional Krylov subspace ${\cal K}_{m-k}(A,r_0)$ and $k<m$ approximate singular vectors. Let $W$ be a matrix of order $n \times m.$ Assume that the first $(m-k)$ columns of $W$ form an orthonormal basis for the Krylov subspace ${\cal K}_{m-k}(A,r_0)$ and 
its last $k$ columns are approximate singular vectors $y_i,$ for $i=1,2,\cdots k.$

The new algorithm recursively constructs first $m-k$ columns of $W$ and an orthonormal basis matrix $Q$ of an $m$ dimensional search subspace. The matrices $W$ and $Q$ satisfy the following relation:
$$AW = Q\tilde{H},$$
where $\tilde{H}$ is an upper Hessenberg matrix of order $(m+1) \times m.$ Note that $Q$ is a matrix of order $n \times (m+1)$ and its first $m-k+1$ columns are formed using the Arnoldi recurrence relation. The last $k$ columns of it are formed by successively orthogonalizing the vectors $Ay_i$ for $i=1,2,.....,k$ against its previous columns. Further, notice that $Q^\ast r_0$ is a multiple of a first coordinate vector.

Similar to the GMRES algorithm, the new algorithm computes an orthogonal matrix $P$ of order $(m+1)$  and an upper
triangular matrix $R$ of order $(m+1) \times m$ such that
$$P\tilde{H} = R.$$ Then, it finds a vector $d$ such that $\|r\|= \|b-A(x0+Wd)\|$ is  minimum, and updates an approximate solution to \^{x}$:= x0+Wd.$ Note that 
$$\|r\| =  \|b-A(x0+Wd)\| = \|r_0-AWd\| = \|r_0-Q\tilde{H}d\|$$
$$ = \|Q^\ast r_0- \tilde{H}d\| = \|PQ^\ast r_0- Rd\|.~~~~~~~~~~~~~~~~~~~$$
%Thus, $\|r\|$ is  minimum and $\|r\|=\|PQ^\ast r_0- Rd\|.$  
%Recall that $Q^\ast r_0$ is a multiple of a first coordinate vector, 
%$P$ is an lower Hessneberg matrix of order $(m+1)$ 
As $R$ is an upper triangular matrix of order $(m+1) \times m$ and $\|PQ^\ast r_0- Rd\|$ is minimum, the new method gives the minimal solution by solving for $d$  that makes the first $m$ entries of $PQ^\ast r_0- Rd$ zero. Hence, $\|r\|$ is equal to the magnitude of the last entry of
$PQ^\ast r_0.$ Therefore, in the new method, $\|r\|$ is a
byproduct and does not require any extra computation.
%as in the GMRES and  the method proposed in \cite {aug}.

Next, We wish to find approximate right singular vectors of $A$  from the
subspace spanned by $W$ to augment the search subspace in the next run. For this, we find eigenvectors corresponding to the $k$ smaller
eigenvalues of the matrix $W^\ast A^\ast AW.$ A little calculation is required to compute this matrix,
because of 
$$G:=W^\ast A^\ast AW = \tilde{H}^
\ast Q^\ast Q  \tilde{H}= \tilde{H}^\ast  \tilde{H}= R^\ast R.$$ We used the  Matlab command "eigs" to
solve the eigenvalue problem for $G.$

The implementation of our new method is as follows.
%: standard GMRES is used, except the singular vector calculations are added at the end. 
For simplicity, a listing of the algorithm has done for the second
and subsequent runs.
\begin{center}
\large One restarted run of GMRES with singular vectors
\end{center}
1. Initial definitions and calculations: The Krylov subspace has
dimension m-k, k is the number of approximate eigenvectors. Let
$q_l= \frac{ r_0}{\|r_0\|}$ and $w_l=q_l.$ Let $y_1, y_2,..., y_k$ be
the approximate singular vectors. Let $W_{m+i}= y_i,$ for $i=
1,2,...,k.$ \\
2. Generation of Arnoldi vectors: For $j= 1, 2,..., m$ do:
\begin{center}
$h_{i,j}= \langle Aq_j, q_i \rangle , ~i= 1, 2,..., j$, \\
$\hat{q}_{j+1} = Aq_j-\displaystyle \sum\limits _{i=1}^{j}
h_{i,j}q_i,$\\ \vs{0.2cm}$h_{j+1,j} = \| \hat{q}_{j+1} \|,$ and \\
\vs{0.2cm} $q_{j+1} = \hat{q}_{j+1} / h_{j+1,j}.$ \\\vs{0.2cm} If $j
< m-k,$ let $w_{j+1}= q_{j+1}.$
\end{center}
3.Addition of approximate singular vectors: For $j=
m-k+1,m-k+2,...m$,do:
\begin{center}
$h_{i,j} = \langle Aw_j, q_i\rangle,$ ~$i= 1, 2,..., j$, \\
$\hat{q}_{j+1} = Aw_j-\displaystyle \sum\limits _{i=1}^{j}
h_{i,j}q_i,$\\
$h_{j+1,j} = \|{\hat{q}_{j+1}}\|,$ and\\\vs{0.2cm}
$q_{j+1} = \frac{\hat{q}_{j+1}}{h_{j+1,j}}.$\\
\end{center}
4. Form the approximate solution: Let $\beta = \|r_0\|.$  Find $d$
that minimizes $\|\beta e_1-\tilde{H}d\|$ for all $d \in {\cal R}^{m}
$. The orthogonal factorization $P\tilde{H}= R$, for R upper
triangular, is used. Then $\hat{x}= xo + Wd.$ \\
5. Form the new approximate singular vectors: Calculate $G= R^\ast
R.$ Solve $Gg_i=\sigma^2 g_i$, for the appropriate $g_i.$ Form $y_i= 
Wg_i$ and $Ay_i= Q \tilde{H}g_i$.\\
6. Restart: Compute $r= b- A\hat{x}$; if satisfied with the residual
norm then stop, else let $x0=\hat{x}$ and go to 2.

%Apart from computation of matrix $G,$ 
Only the Step-5 in the above algorithm is different from the GMRES with eigenvectors method. The GMRES with eigenvectors method requires the
computation of both $F=W^\ast A^\ast W$ and $G$ 
%to obtain eigenvector approximations using Harmonic Rayleigh-Ritz process 
\cite{aug}, whereas the Step-5 in the above algorithm computes the only $G=W^\ast A^\ast AW.$ Hence, the above algorithm requires less computation
and storage compared to the GMRES with eigenvectors method. 
%as it requires computation of the only $G$ but not $F.$ 

Next, we compare the GMRES
with singular vectors and standard GMRES methods. For this, we follow the  procedure that used in \cite{aug} to compare the standard GMRES and GMRES with eigenvectors methods.
%the computation and storage requirements for
It compares only significant expenses. 
%of computation in both the methods.

Suppose the 
search subspace currently at hand is a 
Krylov subspace of dimension $j.$ If the search subspace expands with one more Arnoldi vector, then it requires one matrix-vector product. The orthogonalization requires about $2jn$ multiplications. Instead, if search subspace expanded with a singular vector approximation, no matrix-vector product is required. The other costs are approximately $4jn$ multiplications. It includes $2jn$ multiplications for orthogonalization and $2jn$ for
computing $y_i$ and $Ay_i.$ 
% is expanded
%%the search subspace 
%with the next
%Arnoldi vector, then GMRES requires one matrix-vector product and
%$2jn$ multiplications for the orthogonalization. Instead, if we use
%singular vector approximation, no matrix-vector product is needed.
%The other costs are approximately $4jn$ multiplications. It
%includes $2jn$ multiplications for orthogonalization and $2jn$ for
%computing $y_i$ and $Ay_i.$ 
Hence, GMRES with singular vectors
requires $2jn$ extra multiplications compared to the standard GMRES, but at
the cost of a matrix-vector product in GMRES that requires $n^2$ multiplications. In general, $2j << n.$  Therefore, the GMRES with
singular vectors method requires overall less computation than standard
GMRES.
% and  required to compute matrix-vector product in standard GMRES. 

The GMRES with $k$ singular vector approximations method requires the storage of $m+2k+2$ vectors. This includes the storage of $2k$ vectors, $y_i$ and $Ay_i$ for $i=1,2,\cdots,k.$ However, the storage requirement for  standard GMRES with $m+k$ dimensional Krylov subspace is $m+k+2$ vectors. Thus, the GMRES with singular vectors method requires extra storage compared to standard GMRES. Since $k << m$ this extra storage is often not a problem.
\section{Examples}\vs{-0.2cm} 
In the following, GMRES-SV(m,k) indicates that at each run $k$ approximate singular vectors from the previous run augment Krylov subspace of dimension $m-k.$ Similarly, GMRES-HR(m,k) indicates the augmentation of approximate eigenvectors those obtained using the Harmonic Rayleigh-Ritz process to a Krylov subspace. Moreover, in the first run of both the methods search subspace is a Krylov subspace of dimension $m$ .
%
%In the following, we used the notation GMRES-SV(m,k) to indicate that at each run a $m-k$ dimensional Krylov subspace is augmented by $k$ approximate singular vectors from the previous run. Similarly, GMRES-HR(m,k) indicates that a Krylov subspace is augmented by approximate eigenvectors obtained using Harmonic Rayleigh-Ritz process. Moreover, in both the methods, at the first run,  a search subspace is a Krylov subspace of dimension $m.$ 

In each of the example, we compare GMRES-SV(m,k) with GMRES-HR(m,k), GMRES(m), and GMRES(m+k). Here for GMRES, the number in the  parenthesis represents the dimension of a Krylov subspace at each run. Note that the dimension of search subspaces in GMRES(m), GMRES-SV(m,k) and GMRES-HR(m,k) are same, and in GMRES(m+k) the search subspace at each run requires nearly the same storage as that of GMRES-SV(m,k) and GMRES-HR(m,k).
% Note that in GMRES(m) dimension of a search subspace is same as in the variants GMRES-SV(m,k) and GMRES-HR(m,k). Whereas in GMRES(m+k) the search subspace at each run requires the same storage as that of GMRES-SV(m,k) and GMRES-HR(m,k).

In all numerical examples, the right-hand sides
have all entries $1.0,$ unless mentioned otherwise. The initial
guesses $x0$ are zero vectors. Further, we stopped each algorithm
when $\|r\|/\|b\|$ reduced below the fixed tolerance $10^{-8}.$ All
experiments have been carried out using MATLAB R2016b on intel core i7 system  with $3.40GHZ$ speed. 
\begin{eg}\label{eg1}\relax
This example is same as the example-1 in \cite{errm}. The linear system results from the discretization of one dimensional Laplace equation. The coefficient matrix $A$ is a symmetric tridiagonal matrix of dimension $1000,$ and the right-hand side vector is $(1~ 0~ 0~ \cdots~ 0~ 1)'$. The entry on the main diagonal of $A$ is $2$, whereas $−1$ is on the sub-diagonals of $A$.     The matrix has eigenvalues $\lambda_k= 2.(1-\cos \frac{k\pi}{1001})$ for $1 \leq k \leq 1000$ and its condition number is $(1+\cos \frac{\pi}{1001})/(1-\cos \frac{\pi}{1001}).$
\end{eg}
The Figure-\ref{fig1-B} depicts the convergence of residual norms and corresponding error norms in the GMRES-SV(20,4), GMRES-HR(20,4), GMRES(20), and standard GMRES(24) methods.
\begin{figure}[!htb]
%\begin{minipage}{0.27\linewidth}
\begin{center}
%\epsfxsize=150mm
%\hspace{-1.5cm}\epsfbox[110 85 420 220]{example1-ERRM-kashivishwanathagmresnlog.eps}
\includegraphics[width=5.5in, height=2.5in]
%{crab}[trim= -10 50 -10 50,clip, width=5.5in,height=3.5in]
{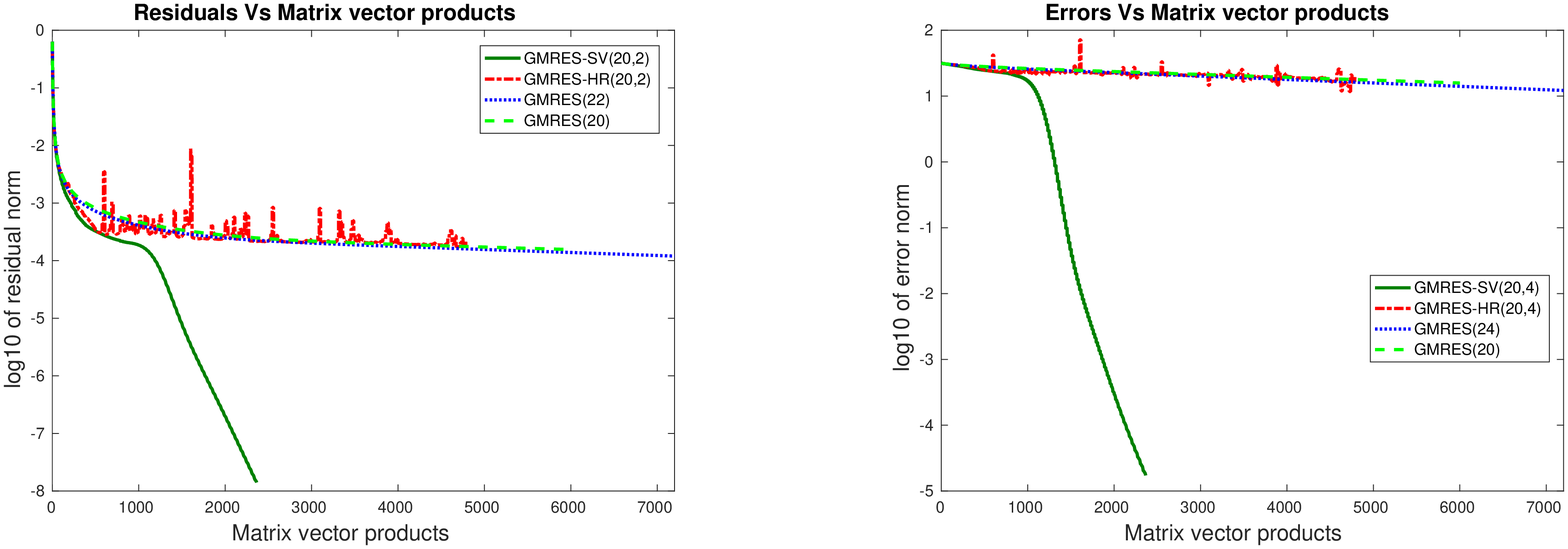}
%{untitled.eps}
\vs{-0.5cm}
\caption{ Magnitudes of $\frac{\|r\|}{ \|b\|}$ with GMRES(20,4) singular vectors/eigenvectors, GMRES(24), and GMRES(20)  (left). Absolute errors with GMRES(20,4) singular vectors/eigenvectors, GMRES(24), and GMRES(20)  (right). }
\label{fig1-B}
\end{center}
\vs{-0.5cm}
%\end{minipage}
\end{figure}

In GMRES-SV(20,4), $\|r\|/\|b\|$ drops to below the tolerance $10^{-8}$ in the $148^{th}$ run. It required $2365$ number of matrix-vector products.  In the remaining three methods $\|r\|/\|b\|$ did not reached at least $10^{-4}$ even after $5000$ matrix-vector products. Here, the total
number of matrix-vector products in all methods counted in a similar way
as in \cite{aug}.

Observe from the right part of Figure-\ref{fig1-B} that GMRES-SV(20,4) reduces error norms also to a far  better extent than the remaining three methods. Here, error norm is the norm of an  error vector, a difference between a solution obtained using "backslash" command in Matlab and an approximate solution in an iterative method.

 From the Figure-\ref{fig1-B}, we observed that when residual norm drops below the tolerance $10^{-8},$ the $\log10$ of an error norm in the GMRES-SV(20,4) is $-4.763.$ In the other three methods, at $5000^{th}$ matrix-vector product it is just near $1.244$. Therefore, this example illustrates the fact that the augmentation of a Krylov subspace with singular vectors reduces error norms and also the residual norms. 

%An error at only residual norms, reduced to a bett 
% the fixed tolerance $10^{-8}$ even after $300$ runs. However, in the Figure-\ref{fig1-B}, we restrict the maximum number of matrix vector product to  It is an easy to observe from it that in the GMRES(20) and GMRES(24) methods $\|r\|/\|b\|$ stays in between $10^{-3}$ and $10^{-4}$ from $500^th$ matrix product onwards. 
 
% the proposed algorithm and the GMRES with eigenvectors. At the first run, the dimension of search subspace in the GMRES with singular vectors/eigenvectors is 20. From the second run onwards, a Krylov subspace of dimension $16$ is augmented by  either $4$ approximate singular vectors/eigenvectors from the previous run. Further, the figure compares augmented GMRES methods with the standard GMRES(20) and GMRES(24).

\begin{eg}\label{sherman}\relax
Consider the matrix $SHERMAN4$ that comes from Oil reservoir modeling. It is a real un-symmetric matrix of order $1104.$  The Matrix market provided the right-hand side vector. We compare GMRES-SV(20,4)
%with singular vectors using $m=20$ and $k=16$ 
(16 Krylov
vectors and 4 approximate right singular vectors) with GMRES-HR(20,4), GMRES(20), and GMRES(24).  
%Thus, in both methods we used subspaces of same size.
\end{eg}\vs{-0.2cm}
See left part of Figure \ref{fig1} for the convergence of $log10$ of residual norms in all the methods.
\begin{figure}[!htb]
%\begin{minipage}{0.27\linewidth}
\begin{center}
%\epsfxsize=150mm
%\hspace{-1.5cm}\epsfbox[110 85 420 220]
\includegraphics[width=5.5in,height=2.5in]
{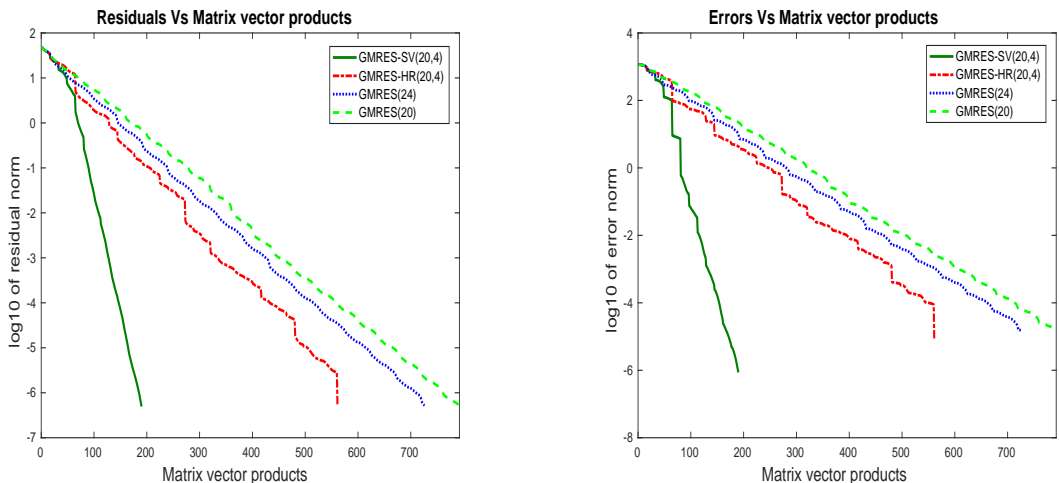}
\vs{-0.5cm}
\caption{ Magnitudes of $\frac{\|r\|}{ \|b\|}$ with GMRES(20,4) singular vectors/eigenvectors, GMRES(24), and GMRES(20)  (left). Absolute errors with GMRES(20,4)  singular vectors/eigenvectors, GMRES(24), and GMRES(20)  (right). }
\label{fig1}
\end{center}
\vs{-0.5cm}
%\end{minipage}
\end{figure}

In GMRES-SV(20,4) the quantity $\|r\|/\|b\|$
reduced to below $10^{-8}$ at the $12^{th}$ run.  
%After 12 runs, the $\log10$ of residual norm is $-6.305,$ 
%%$4.958619962275146e-007$, 
%whereas  $\log10$ of the ratio
%$\|r\|/\|b\|$ is $-8.0248.$
%%$9.445079426163018e-009.$ 
The total number of matrix-vector products it required is $190.$  GMRES-HR(20,4) required $562$ matrix vector products to drop $\|r\|/\|b\|$ below the tolerance $10^{-8}.$ 
%took $36$ runs and
Thus, GMRES-HR had required nearly thrice the computation than the GMRES-SV method. Further, observe from the Figure-\ref{fig1} that
%convergence of residual norms in GMRES(20) and GMRES(24) is slower than GMRES-HR. Therefore, 
GMRES-SV(20,4) is far better than GMRES(24) even though it used smaller search subspaces.

The right part of the Figure-\ref{fig1} compares error norms.
% in all the methods. 
When the residual norm reached the tolerance, the $\log10$ of an error norm in GMRES-SV(20,4)is $-6.063,$ whereas it is $-5.063 $ in GMRES-HR(20,4), and  is equal to $-4.813,$ $-4.802$ in the GMRES(24) and GMRES(20) methods respectively. Therefore, for this example, the GMRES-SV method significantly reduced the error norm compared to the remaining three methods.
%
%In the GMRES(20) and GMRES(24) have used too
% GMRES method took $40 $ runs to reduce the ratio
%$\|r\|/\|b\|$ to the desired tolerance. 
% 
%GMRES with singular vectors uses $190$ number
%of matrix products, where as GMRES(20) used $792$ number of
%matrix-vector products.
%%%include figures here for comparison of errors and residuals.
%%%discuss about errors here.
%
%We also compare the error norms in both GMRES(20) and GMRES with
%singular vector methods. The new method, GMRES with singular vector
%effectively reduces the norm of error compared to the GMRES(20).
%% In both methods, error vector is a difference between a solution
%%obtained using "back slash" command in Matlab and an approximate
%%solution in iterative method.
% See Fig:\ref{fig1} for the comparison
%of error norms. GMRES with singular vectors reduces the norm of
%error up to the order $10^{-7},$ where as GMRES(20) reduces the
%error norm only  up to the order $10^{-1}.$
%
%Next, methods requiring about the same storage are compared.   In
%Figure-\ref{fig1},  we
%compared the results of GMRES with singular vectors method and GMRES(24). With GMRES(24),
%norm of error reduced to $9.992708081382941e-009$ after $30$ runs and it required $726$
%number of matrix-vector products. Thus, GMRES with singular vectors
%method for $m=20$ and $k=16$ is still better than GMRES(24). Moreover,  it used smaller subspaces and far-less matrix vector products.
\begin{eg}\label{watt}\relax
Consider the matrix $WATT1$ that came from petroleum engineering. It
is a real un-symmetric matrix of order $1856$ with $11360$ non-zero
entries. The right-hand side is the one provided by the Matrix
Market. To see the performance of GMRES-SV with fewer approximate
singular vectors, we have chosen $m=20$ and $k=2.$ Thus, we have
used only two approximate singular vectors, which are less in number
compared to the previous examples. 
\end{eg}
\begin{figure}[!htb]
%\begin{minipage}{0.27\linewidth}
\begin{center}
%\epsfxsize=150mm
%\hspace{-1.5cm}\epsfbox[110 85 420 220]
\includegraphics[width=5.5in,height=2.5in]
{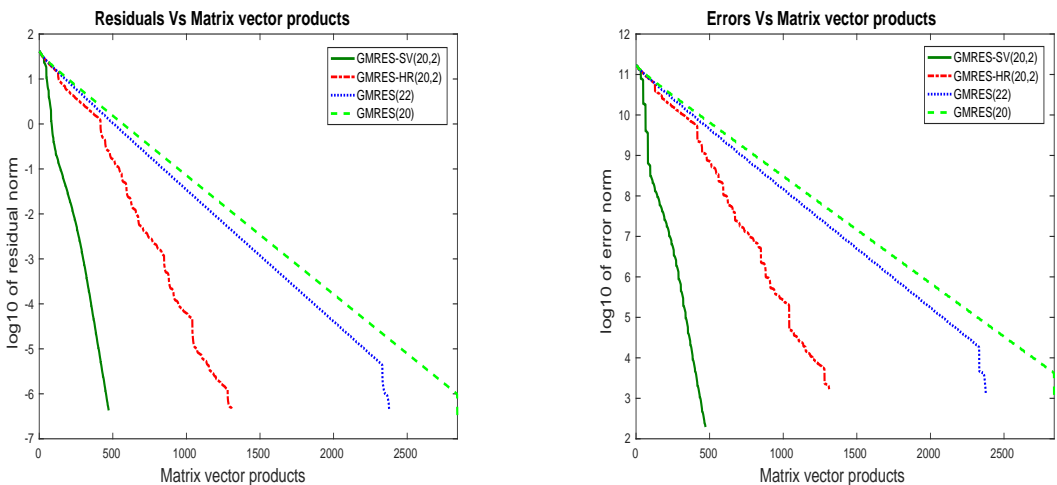}
\vs{-0.5cm}
\caption{ Magnitudes of $\frac{\|r\|}{ \|b\|}$ with GMRES(20,2)  singular vectors/eigenvectors, GMRES(22), and GMRES(20)  (left). Absolute errors with GMRES(20,2)  singular vectors/eigenvectors, GMRES(22), and GMRES(20)  (right). }
\label{fig2}
\end{center}
%\end{minipage}
\vs{-0.9cm}
\end{figure}
Using GMRES-SV(20,2), the
ratio $\|r\|/\|b\|$ reached the required tolerance in the $30^{th}$ run, whereas in GMRES-HR(20,2), GMRES(22), and GMRES(20) it happened in $82^{nd},109^{th},$ and $143^{rd}$ run respectively. See
Figure-\ref{fig2}(left) for the comparison of $\log10$ of residual norms in
all the four methods.

Figure-\ref{fig2}(right), compares the convergence of error norms in four methods. Observe from it that GMRES-SV(20,2) reduced the error norm
to a better extent compared to the other three methods, even
though it took fewer iterations for the convergence of $\|r\|/\|b\|.$ Also, it reduced residual norms as well.
%%Using GMRES with singular vectors, norm of error at the end of
%%$55^{th}$ run is $1.161484372333233e-004.$ It not reduced the error
%%further since the residual norms reduced to the order of
%%$10^{-13}$ at the end of $55^{th}$ run.
%\vs{0.85cm}\\
%In the next example, we are comparing our new method GMRES with
%singular vectors, Standard GMRES and GMRES with eigenvectors
%proposed by Morgan in \cite{aug}. For this, we considered the matrix
%in Example-1 of the paper \cite{aug}.\vs{-0.3cm}
\begin{eg}\label{mor}\relax
This example has taken from \cite{aug}. It is a bidiagonal matrix of order $1000.$ The diagonal elements
are $1,2, \cdots, 1000$ in order. The super diagonal elements are
$0.1s.$ 
% For our new method, 
We have chosen $m= 20$ and $k=2$ for
GMRES method with singular vectors. We used only two eigenvector
approximations in GMRES-HR. We compare these two methods with GMRES(20) and GMRES(22).
%Subspace size in Standard
%GMRES is $20.$ Thus, we are using subspaces of same size in all
%methods.
\end{eg}
\begin{figure}[!htb]\relax
%\begin{minipage}{0.27\linewidth}
\begin{center}
%\epsfxsize=150mm
%\hspace{-1.5cm}
%\epsfbox[110 85 420 220]
\includegraphics[width=5.5in,height=2.5in]
{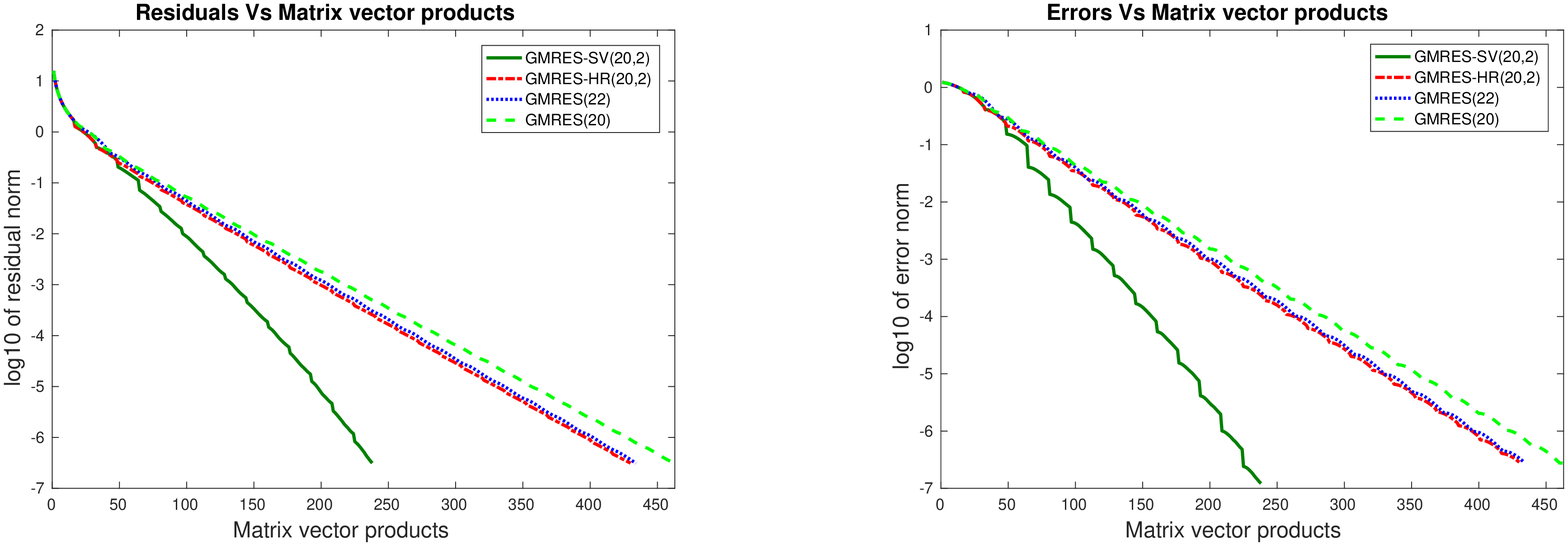}
%\vs{0.5cm}
\caption{ Magnitudes of $\frac{\|r\|}{ \|b\|}$ with GMRES(20,2) singular vectors/eigenvectors, GMRES(22), and GMRES(20)(left). Absolute errors with GMRES(20,2)  singular vectors/eigenvectors, GMRES(22), and GMRES(20) (right). }
\label{fig3}
\end{center}
\vs{-0.5cm}
%\end{minipage}
\end{figure}
Figure-\ref{fig3}(left), compares a ratio $\|r\|/\|b\|$ in these
four different methods. Using GMRES-SV(20,2)
$\|r\|/\|b\|$ had reached the desired tolerance in a $15^{th}$ run,
whereas in GMRES-HR(20,2), GMRES(22), and GMRES(20), this ratio reduced below the tolerance in $27^{th},20^{th},$ and $24^{th}$ run, respectively.
%include figure here.
In the Figure-\ref{fig3}(right), we compared the error norms in the four
methods. Though the error is reduced up to the same order in all methods, the GMRES-SV(20,2) has taken less number of matrix-vector products. Further, observe that in GMRES-SV smaller error norm at each iteration accelerates the convergence of residual norms .
% All methods reduced the error up to the same order when
%the residual norm reached the required tolerance. But, observe that GMRES-SV(20,2) has taken less number  of matrix-vector products for the convergence. Further, observe that a smaller error norm at each iteration in the GMRES-SV method accelerates the convergence of residual norms in it.

Above examples have shown that our new method is effective in accelerating the convergence of GMRES. It also shows that we can use
singular vector approximations instead of eigenvector approximations
to augment the search subspace. Further, example-1 has shown the superiority of the GMRES-SV method even in the case of near stagnation of error norms in  standard GMRES. 

We reported four typical examples in detail though computation carried out on several matrices available in the Matrix Market. The Table-1 reports a summary of results on eight other matrices with various base sizes.
% when the convergence is attained. 
It is apparent from the table that the GMRES with singular vectors method performs better in reducing the error norms compared to standard GMRES and GMRES-HR.
% as it requires a lesser computation; For details see the Section-4.
%number of matrix-vector multiplications (MVP).

%Though computation has been performed on many matrices available in
%the matrix market, we have reported four typical cases in detail. The Table-1
%reports summary of results on eight other matrices with various base sizes until
%convergence is attained. It is clear that the performance of GMRES with singular vectors  is better in comparison  with the standard GMRES and GMRES-HR in any particular example; and moreover, it converges with less number of matrix-vector multiplications (MVP).
\begin{table}[!htb]
    \begin{tiny}
    \begin{center}
        \caption{Summary results on other matrices}
        \begin{tabular}{|c|c|c|c|c|c|c|}
            \hline
Matrix& Method &rhs&Initial vector &MVP&$\|r\|/\|b\|$&error \\ \hline
Add20 & GMRES-SV(30,4) &   NIST &  Zeros(2395,1) & 17*26 & 8.903030927585648e-09 & 6.249677395458858e-15\\
     &GMRES(30) & &&        27*30+23&   9.951915432072370e-09& 1.269588940989828e-14\\
          &GMRES-HR(30,4) & &&24*26+9        &9.859393321942817e-09   & 1.326769652866634e-14
 \\&& &&        &   & \\
Bcsstm12&   GMRES-SV(30,4)&    Ones(1473,1)&   Zeros(1473,1)&  7*26+21&  9.735591826540923e-09& 4.010655907680425e-05
\\
    &GMRES(30)  &&&  7*30+18&   8.504590396065243e-09& 3.371368967798073e-05\\
              &GMRES-HR(30,4) & &&9*26+4      &9.359680271393492e-09   & 3.454242536809857e-05 \\&& &&        &   & \\
Cavity05    &GMRES-SV(30,4)    &NIST   &Zeros(1182,1)& 58*26+3    &9.979063739097918e-09  &    8.864891034548036e-17\\
    &GMRES(30)* &&&     300*30&9.083509763407927e-06 &4.366773990806636e-12\\
              &GMRES-HR(30,4)* & &&  300*30     &4.188013345206927e-04   & 2.098667819011905e-10 \\&& &&        &   & \\
Cavity10&   GMRES-SV(30,4)&    NIST    &Zeros(2597,1)& 107*26+2&  9.748878041774344e-09&1.082459456375754e-14\\
    &GMRES(30)* &&&     300*30&  1.480273732125253e-05  &1.745590861026809e-09\\
              &GMRES-HR(30,4)* & && 300*30       & 4.412649085564300e-05  & 5.202816341776465e-09 \\&& &&        &   & \\
Cdde1&  GMRES-SV(30,4)&    Ones(961,1) &Zeros(961,1)&  7*26+4 &   9.375985059553094e-09  &9.904876436183557e-06\\
    &GMRES(30)      &&& 28*30+24&   9.903484254958585e-09
  &     8.057216911709931e-05\\
              &GMRES-HR(30,4) & &&42*26  & 9.784366641015672e-09
  &5.754305524376241e-05   \\&& &&        &   & \\
$Orsreg_1$  &GMRES-SV(30,4)&   Ones(2205,1)&   Zeros(2205,1)&  9*26+18&    9.636073229113061e-09 & 3.069738625094621e-08 \\
    &GMRES(30)&&&           13*30+18&   9.671973942415371e-09& 8.655424850905282e-08\\
         &GMRES-HR(30,4) & && 15*26+24  & 9.858474091767057e-09  &7.314829956508281e-08 \\&& &&        &   & \\
Sherman1    &GMRES-SV(30,4)    &NIST&  Zeros(1000,1)&  34*26+16 &9.988703017482971e-09&    3.073177766807258e-05 \\
    &GMRES(30)  &&&     103*30+21&   9.987479947720702e-09& 1.160914298556447e-04\\
              &GMRES-HR(30,4) & && 35*26+1 &8.658352412777792e-09 & 5.119874131512304e-05 \\&& &&        &   & \\
$Watt_2$    &GMRES-SV(30,4)&   Ones(1856,1)&   Zeros(1856,1)&  40*26    &9.956725136180967e-09&    1.955980343013359e+03\\
    &GMRES(30)  &&&     168*30+6  &9.897399520154960e-09
&    6.004185029785801e+03\\
              &GMRES-HR(30,2) & &&254*26  &9.996375252736734e-09
   & 7.072952057904888e+03  \\
 \hline
\end{tabular}
\end{center}
\end{tiny}
\end{table}
In Table-1, NIST refers to the right-hand side vector provided by Matrix Market website and GMRES* represents the non-convergence of the GMRES method even after 300 iterations. Moreover, for counting the number of matrix-vector products(MVP) we followed the same procedure as in  \cite{aug}.  In the above table $x*y+z$ means in each of the $x$ iterations, the specific method used $y$ MVPs and in the $(x+1)^{th}$ iteration, it used $z$ Matrix-Vector Products.

\section{Conclusions}
In this paper, a new augmentation procedure in GMRES has been proposed using approximate right singular vectors of a coefficient matrix. The proposed  method has an advantage that it requires less computation compared to the GMRES with Harmonic Ritz vectors method. Unlike the augmentation method in \cite{aug}, the proposed method reduces the error norms also to a better extent. 
%This, we have proved in Theorem. In addition,
Further, the proposed method involves the computation in real arithmetic  for the  matrices and right-hand side vectors in the real number system. Numerical experiments have been carried out on benchmark matrices. Results have shown the superiority of the proposed method over the standard GMRES and  GMRES with Harmonic Ritz vectors methods.
%\textcolor{red}{ However, except in the last example, we compared the proposed method only with standard GMRES, since the GMRES with Harmonic Ritz vectors came from the context different from this paper.}

\section*{Acknowledgements} 
The author thanks the National Board of Higher Mathematics, India for supporting this work under the Grant number2/40(3)/2016/R\&D-II/9602 .

\end{document}